\documentclass[11pt]{article}
\usepackage{amssymb,amsmath,amsthm,amsfonts}
\usepackage{graphicx}
\usepackage{epsfig}

\usepackage{amssymb,amsmath,amsthm,amsfonts,mathrsfs, bbm,dsfont}

\theoremstyle{plain}

\theoremstyle{definition}
\newtheorem{definition}{Definition}[section]

\newtheorem{remark}{Remark}[section]

\newtheorem{theorem}{Theorem}[section]
\newtheorem{lemma}{Lemma}[section]

\newtheorem{corollary}{Corollary}[section]

\numberwithin{equation}{section}

\begin{document}
\openup 0.8\jot
\title{\Large\bf Common coupled fixed point theorems
in $C^*$-algebra-valued metric spaces \thanks{This work is supported by the National Natural Science Foundation of China
(Grants Nos. 61701343 and 11701423)} }
\author{Cao Tianqing$^1$ \thanks{E-mail address:caotianqing@tiangong.edu.cn}, Xin Qiaoling$^2$}
\date{}
\maketitle\begin{center}
\begin{minipage}{16cm}
{\small \it
$1$
School of Mathematical Sciences, Tiangong University, Tianjin 300387, China

$2$
School of Mathematical Sciences, Tianjin Normal University, Tianjin 300387, China

}
\end{minipage}
\end{center}
\vspace{0.05cm}
\begin{center}
\begin{minipage}{16cm}
{\small {\bf Abstract}:
In this paper, we prove some common coupled fixed point theorems for mappings satisfying different contractive
conditions in the context of complete $C^*$-algebra-valued metric spaces.
Moreover,
the paper provides an application to prove the existence and uniqueness of a
solution for Fredholm nonlinear integral equations.}
\endabstract
\end{minipage}\vspace{0.10cm}
\begin{minipage}{16cm}
\end{minipage}
\end{center}
\begin{center} \vspace{0.01cm}
\end{center}

\section{Introduction}
Fixed point theory, one of the active research areas in mathematics, focuses on maps and
abstract spaces. For example, fixed point theorems found
important applications to study the existence and uniqueness of solutions
for matrix equations, ordinary differential and integral equations, see \cite{Aydi,Bhaskar,Jleli,Nieto07,Ran} and references therein.

The notion of coupled fixed points was introduced by Guo and Lakshmikantham \cite{Guo}. Since then, the concept has been of interest to many researchers in fixed point
theory.
In 2006, Bhaskar and Lakshmikantham \cite{Bhaskar} introduced the concept of a mixed monotone property for the first time and investigated some coupled fixed point theorems for mappings. They also discussed the existence and uniqueness of solutions for the periodic boundary value problem as an application of their result.
Afterward, Sabetghadam et al. \cite{Sabetghadam}
introduced this concept in cone metric spaces and proved some fixed point theorems in cone metric spaces.
Later in \cite{Luong} Luong
and Thuan studied the existence and uniqueness of solutions for nonlinear integral equations as an application of coupled fixed points. Subsequently, Jleli and Samet \cite{Jleli} discussed the existence and uniqueness of a positive solution
for a class of singular nonlinear fractional differential equations.
As a result, many authors obtained many coupled fixed point and coupled coincidence
theorems in ordered metric spaces \cite{Aydi,Beg,Berinde,Samet}.

On the other hand, many authors studied the fixed and coupled fixed point theorems for different spaces, like in $b$-metric spaces \cite{Malhotra}, cone metric spaces \cite{Huang}, fuzzy metric spaces \cite{Mihet}, $G$-metric space \cite{Chen}, quasi-Banach spaces \cite{Hussain}, noncommutative Banach spaces \cite{Xin}, and so on.
In 2007, Huang and Zhang \cite{Huang} introduced cone metric spaces which generalized metric spaces, and obtained various fixed point theorems for contractive mappings.
Afterward, many authors investigated coupled fixed point their fixed point theorems in cone metric spaces \cite{Abbas,Karapinar10,Sabetghadam}. In 2014, Ma et al. \cite{Ma} initially introduced the concept
of $C^*$-algebra-valued metric spaces, and proved some fixed point theorems for
self-maps with contractive or expansive conditions on such spaces.

Motivated by the
works of Abbas et al. \cite{Abbas} and Ma et al. \cite{Ma}, in the present paper, we shall prove corresponding
common coupled fixed point theorems in $C^*$-algebra-valued metric spaces. More precisely, we
prove some common coupled fixed point theorems for the mapping under different contractive conditions.
We also illustrate how our results can be
applied to obtain the existence and uniqueness results
for Fredholm nonlinear integral equations.

First of all, we recall some basic definitions, notations and results of $C^*$-algebra that can be found in \cite{Murphy}.
Let $\mathcal A$ be a unital algebra. An involution on $\mathcal A$ is a conjugate-linear map $a\rightarrow a^*$ on $\mathcal A$ such that $a^{**}=a$ and $(ab)^*=b^*a^*$ for any $a,b\in \mathcal A$. The pair $(\mathcal A, *)$ is called a $*$-algebra. A $*$-algebra $\mathcal A$ together with a complete submultiplicative norm such that $\|a^*\|=\|a\|$ is said to be a Banach $*$-algebra. Furthermore, A $C^*$-algebra is a Banach $*$-algebra with $\|a^*a\|=\|a\|^2$, for all $a\in \mathcal A$. An element $a$ of a $C^*$-algebra ${\mathcal A}$ if positive if $a$ is hermitian and $\sigma(a)\subseteq [0,+\infty)$, where $\sigma(a)$ is the spectrum of $a$. We write $0_{\mathcal A}\preceq a$ to show that $a$ is positive, and denote by ${\mathcal A}_+$, ${\mathcal A}_h$ the set of positive elements, hermitian elements of ${\mathcal A}$, respectively, where $0_{\mathcal A}$ is the zero element in $\mathcal A$. There is a natural partial ordering on $\mathcal A_h$ given by
$a \preceq b$ if and only if $0_{\mathcal A}\preceq b-a$. From now on, ${\mathcal A}'$ will denote the set  the set $\{a\in \mathcal A\colon ab=ba, \forall b\in\mathcal A\}$.

Before giving our main results, we recall some basic concepts and results which will be needed in what follows. For more details, one can see \cite{Ma}.

\begin{definition}
Let $X$ be a nonempty set. Suppose that the mapping $d\colon X\times X\rightarrow \mathcal A$ is defined, with the following properties:
\begin{enumerate}
\item[(1)]
$0_{\mathcal A}\preceq d(x,y)$ for all $x, y$ in $X$;
\item[(2)]
$d(x,y)=0_{\mathcal A}$ if and only if $x=y$;
\item[(3)]
$d(x, y)=d(y, x)$ for all $x$ and $y$ in $X$;
\item[(4)]
$d(x, y)\preceq d(x, z)+d(z, y)$ for all $x$, $y$ and $z$ in $X$.
\end{enumerate}
Then $d$ is said to be a $C^*$-algebra-valued metric on $X$, and $(X,{\mathcal A},d)$ is said to be a $C^*$-algebra-valued metric space.
\end{definition}

\begin{definition}
Suppose that $(X,{\mathcal A},d)$ is a $C^*$-algebra-valued metric space. Let $\{x_n\}_{n=1}^\infty$ be a sequence in $X$ and $x\in X$. If
$d(x_n, x)\overset{\|\cdot\|_{\mathcal A}}{\longrightarrow} 0_{\mathcal A}$ $(n\rightarrow\infty)$, then it is said that $\{x_n\}$ converges to $x$, and we denote it by $\lim\limits_{n\rightarrow\infty}x_n=x$.
If for any $p\in\Bbb{N}$,
$d(x_{n+p}, x_n)\overset{\|\cdot\|_{\mathcal A}}{\longrightarrow} 0_{\mathcal A}$ $(n\rightarrow\infty)$, then $\{x_n\}$ is called a
Cauchy sequence in $X$.
\end{definition}

If every Cauchy sequence is convergent in $X$, then
$(X,{\mathcal A},d)$ is called a complete $C^*$-algebra-valued metric space.

It is obvious that any Banach space must be a complete $C^*$-algebra-valued metric space. Moreover,
$C^*$-algebra-valued metric spaces generalize normed linear spaces and metric spaces. There are some non-trivial examples
of complete $C^*$-algebra-valued metric spaces \cite{Ma}.

\section{Main results}
In this section we shall prove some common coupled fixed point theorems for different contractive mappings
in the setting of $C^*$-algebra-valued metric spaces.
Firstly we recall that the definitions which will be needed.

\begin{definition}\label{def1}\cite{Abbas}
Let $X$ be a non-empty set.
An element $(x,y)\in X\times X$ is said to be

(1) a couple fixed point of the mapping $F\colon X\times X\rightarrow X$ if $F(x,y)=x$ and $F(y,x)=y$.

(2) a coupled coincidence point of the mappings $F\colon X\times X\rightarrow X$ and $g\colon X\rightarrow X$ if
$F(x,y)=gx$ and $F(y,x)=gy$. In this case, $(gx,gy)$ is said to be coupled point of coincidence.

(3) a common coupled fixed point of the mappings $F\colon X\times X\rightarrow X$ and $g\colon X\rightarrow X$ if
$F(x,y)=gx=x$ and $F(y,x)=gy=y$.
\end{definition}

Note that Definition \ref{def1} (3) reduces to Definition \ref{def1} (1) if the mapping $g$ is the identity mapping.

\begin{definition}\cite{Abbas}
The mappings $F\colon X\times X\rightarrow X$ and $g\colon X\rightarrow X$ is said to be w-compatible if $g(F(x,y))=F(gx,gy)$ whenever $gx=F(x,y)$ and $gy=F(y,x)$.
\end{definition}

Now we give our main results.

\begin{theorem}\label{Th1}
Let $(X,\mathcal{A},d)$ be a complete $C^*$-algebra-valued metric space.
Suppose that the mappings $F\colon X\times X\rightarrow X$ and $g\colon X\rightarrow X$ satisfy the following condition
\begin{equation}\label{2-1}
d(F(x,y),F(u,v))\preceq a^*d(gx,gu)a+a^*d(gy,gv)a, \ \mbox{for\ any}\ x,y,u,v\in X,
\end{equation}
where $a\in \mathcal{A}$ with $\|a\|<\frac{1}{\sqrt{2}}$.
If $F(X\times X)\subseteq g(X)$ and $g(X)$ is complete in $X$, then $F$ and $g$ have a
coupled coincidence point. Moreover, if $F$ and $g$ are w-compatible, then they have unique common
coupled fixed point in $X$.
\end{theorem}

\begin{proof}
Take $x_0$, $y_0$ in $X$, and let $g(x_1)=F(x_0,y_0)$ and $g(y_1)=F(y_0,x_0)$. One can obtain two sequences $\{x_n\}$ and $\{y_n\}$ by continuing this process such that $g(x_{n+1})=F(x_n,y_n)$ and
$g(y_{n+1})=F(y_n,x_n)$. From (\ref{2-1}), we get
\begin{eqnarray}\label{2-1-1}
    \begin{array}{rcl}
d(gx_n,gx_{n+1})&=&d(F(x_{n-1},y_{n-1}),F(x_n,y_n))\\[5pt]
      &\preceq&a^*d(gx_{n-1},gx_n)a+a^*d(gy_{n-1},gy_n)a\\[5pt]
      &\preceq&a^*(d(gx_{n-1},gx_n)+d(gy_{n-1},gy_n))a.
 \end{array}
\end{eqnarray}

Similarly,
\begin{eqnarray}\label{2-1-2}
    \begin{array}{rcl}
d(gy_n,gy_{n+1})&=&d(F(y_{n-1},x_{n-1}),F(y_n,x_n))\\[5pt]
      &\preceq&a^*d(gy_{n-1},gy_n)a+a^*d(gx_{n-1},gx_n)a\\[5pt]
      &\preceq&a^*(d(gy_{n-1},gy_n)+d(gx_{n-1},gx_n))a.
 \end{array}
\end{eqnarray}

Let
\begin{eqnarray*}
\delta_n=d(gx_n,gx_{n+1})+d(gy_n,gy_{n+1}),
\end{eqnarray*}
and now from (\ref{2-1-1}) and (\ref{2-1-2}), we have
\begin{equation*}
\begin{array}{rcl}
\delta_n&=&d(gx_n,gx_{n+1})+d(gy_n,gy_{n+1})\\[5pt]
      &\preceq&a^*(d(gx_{n-1},gx_n)+d(gy_{n-1},gy_n))a
      +a^*(d(gy_{n-1},gy_n)+d(gx_{n-1},gx_n))a\\[5pt]
      &\preceq&(\sqrt{2}a)^*(d(gx_{n-1},gx_n)+d(gy_{n-1},gy_n))
      (\sqrt{2}a)\\[5pt]
      &\preceq&(\sqrt{2}a)^*\delta_{n-1}(\sqrt{2}a),
\end{array}
\end{equation*}
which, together with the property: if $b,c\in {\mathcal A}_h$, then $b\preceq c$ implies $a^*ba\preceq a^*ca$ (Theorem 2.2.5 in \cite{Murphy}), yields that for each $n\in\Bbb{N}$,
\begin{equation*}
0_{\mathcal A}\preceq\delta_n
      \preceq(\sqrt{2}a)^*\delta_{n-1}(\sqrt{2}a)
      \preceq \cdots
       \preceq[(\sqrt{2}a)^*]^{n}\delta_0(\sqrt{2}a)^{n}.
\end{equation*}

If $\delta_0=0_{\mathcal A}$, then from Definition 1.1 (2) we know that $F$ and $g$ have a coupled coincidence point $(x_0,y_0)$.
Now, letting $0_{\mathcal A}\preceq\delta_0$,
we can obtain for $n\in\Bbb{N}$ and any $p\in\Bbb{N}$,
\begin{eqnarray*}
\begin{array}{rcl}
d(gx_{n+p},gx_{n})&\preceq& d(gx_{n+p},gx_{n+p-1})+d(gx_{n+p-1},gx_{n+p-2})
              +\cdots+d(gx_{n+1},gx_{n}),\\[5pt]
d(gy_{n+p},gy_{n})&\preceq& d(gy_{n+p},gy_{n+p-1})+d(gy_{n+p-1},gy_{n+p-2})
              +\cdots+d(gy_{n+1},gy_{n}).
 \end{array}
 \end{eqnarray*}
Consequently,
\begin{eqnarray*}
    \begin{array}{rcl}
d(gx_{n+p},gx_{n})+d(gy_{n+p},gy_{n})
&\preceq&\delta_{n+p-1}+\delta_{n+p-2}+\cdots+\delta_{n}\\[5pt]
&\preceq&\sum\limits_{k=n}^{n+p-1}[(\sqrt{2}a)^*]^k\delta_0(\sqrt{2}a)^k,
 \end{array}
 \end{eqnarray*}
and then
\begin{eqnarray*}
\|d(gx_{n+p},gx_{n})+d(gy_{n+p},gy_{n})\|
\leq\sum\limits_{k=n}^{n+p-1}\|\sqrt{2}a\|^{2k}\delta_0
\leq\sum\limits_{k=n}^{\infty}\|\sqrt{2}a\|^{2k}\delta_0
=\frac{\|\sqrt{2}a\|^{2n}}{1-\|\sqrt{2}a\|^{2}}\delta_0.
 \end{eqnarray*}
Since $\|a\|<\frac{1}{\sqrt{2}}$, we have
\begin{eqnarray*}
\|d(gx_{n+p},gx_{n})+d(gy_{n+p},gy_{n})\|
\leq\frac{\|\sqrt{2}a\|^{2n}}{1-\|\sqrt{2}a\|^{2}}\delta_0\rightarrow 0,
 \end{eqnarray*}
which, together with $d(gx_{n+p},gx_{n})\preceq d(gx_{n+p},gx_{n})+d(gy_{n+p},gy_{n})$ and $d(gy_{n+p},gy_{n})\preceq d(gx_{n+p},gx_{n})$ $+d(gy_{n+p},gy_{n})$, implies that $\{gx_n\}$ and $\{gy_n\}$ are Cauchy sequences in $g(X)$. Since $g(X)$ is complete, there exist $x,y\in X$ such that $\lim\limits_{n\rightarrow\infty}gx_n=gx$ and $\lim\limits_{n\rightarrow\infty}gy_n=gy$. Now we prove that
$F(x,y)=gx$ and $F(y,x)=gy$. For that we have
\begin{eqnarray*}
    \begin{array}{rcl}
d(F(x,y),gx)
&\preceq&d(F(x,y),gx_{n+1})+d(gx_{n+1},gx)\\[5pt]
&\preceq&d(F(x,y),F(x_n,y_n))+d(gx_{n+1},gx)\\[5pt]
&\preceq&a^*d(gx_n,gx)a+a^*d(gy_n,gy)a+d(gx_{n+1},gx).
\end{array}
 \end{eqnarray*}
Taking the limit as $n\rightarrow\infty$ in the above relation, we get $d(F(x,y),gx)=0_{\mathcal A}$ and hence $F(x,y)=gx$.
Similarly, $F(y,x)=gy$. Therefore, $F$ and $g$ have a coupled coincidence point $(x,y)$.

Now if $F$ and $g$ have a coupled coincidence point $(x',y')$, then
\begin{eqnarray*}
d(gx,gx')=d(F(x,y),F(x',y'))\preceq a^*d(gx,gx')a+a^*d(gy,gy')a,\\[5pt]
d(gy,gy')=d(F(y,x),F(y',x'))\preceq a^*d(gy,gy')a+a^*d(gx,gx')a,
 \end{eqnarray*}
and hence
\begin{eqnarray*}
d(gx,gx')+d(gy,gy')\preceq (\sqrt{2}a)^*(d(gx,gx')+d(gy,gy'))(\sqrt{2}a),
 \end{eqnarray*}
which further induces that
\begin{eqnarray*}
\|d(gx,gx')+d(gy,gy')\|\leq \|\sqrt{2}a\|^2\|d(gx,gx')+d(gy,gy')\|.
 \end{eqnarray*}
Since $\|\sqrt{2}a\|< 1$, then $\|d(gx,gx')+d(gy,gy')\|=0$. Hence we get $gx=gx'$ and $gy=gy'$. Similarly, we can prove $gx=gy'$ and $gy=gx'$. Then $F$ and $g$ have a unique coupled point of coincidence $(gx,gx)$. Moreover, set $v=gx$, then $v=gx=F(x,x)$. Since $F$ and $g$ are w-compatible,
\begin{eqnarray*}
gv=g(gx)=g(F(x,x))=F(gx,gx)=F(v,v),
\end{eqnarray*}
which means that $F$ and $g$ have a coupled point of coincidence
$(gv,gv)$. By the uniqueness, we know $gv=gx$, which yields that
$v=gv=F(v,v)$. Therefore $F$ and $g$ have a unique common coupled fixed point $(v,v)$.
\end{proof}

\begin{remark}
Let $X=\Bbb{R}$ and $\mathcal{A}=M_2(\Bbb{C})$ and the map $d\colon X\times X\rightarrow \mathcal{A}$ is defined by
\begin{eqnarray*}
d(x,y)=\left[
          \begin{array}{cc}
            |x-y| & 0 \\
            0 & k|x-y| \\
          \end{array}
        \right],
         \end{eqnarray*}
where $k>0$ is a constant.
Then $(X,\mathcal{A},d)$ is a complete $C^*$-algebra-valued metric space.
Consider the mappings $F\colon X\times X\rightarrow X$ with $F(x,y)=\frac{x+y}{2}$ and $g\colon  X\rightarrow X$ with $g(x)=2x$.
Set $\lambda\in\Bbb{C}$ with $|\lambda|<\frac{1}{\sqrt{2}}$, and $a=\left[
          \begin{array}{cc}
            \lambda & 0 \\
            0 & \lambda \\
          \end{array}
        \right]$,
then $a\in \mathcal{A}$ and $\|a\|_\infty=|\lambda|$. Moreover, one can verify that $F$ satisfies the contractive condition
\begin{equation*}
d(F(x,y),F(u,v))\preceq a^*d(x,u)a+a^*d(y,v)a, \ \mbox{for\ any}\ x,y,u,v\in X.
\end{equation*}
In this case, $(0,0)$ is coupled coincidence point of $F$ and $g$. Moreover, $(0,0)$ is a unique common couple fixed point of $F$ and $g$ since they are w-compatible.
\end{remark}

\begin{corollary}\label{corollary1}
Let $(X,\mathcal{A},d)$ be a complete $C^*$-algebra-valued metric space.
Suppose that the mapping $F\colon X\times X\rightarrow X$ satisfies the following condition
\begin{equation}\label{2-1}
d(F(x,y),F(u,v))\preceq a^*d(x,u)a+a^*d(y,v)a, \ \mbox{for\ any}\ x,y,u,v\in X,
\end{equation}
where $a\in \mathcal{A}$ with $\|a\|<\frac{1}{\sqrt{2}}$.
Then $F$ has a unique
coupled fixed point.
\end{corollary}

Before going to another theorem, we recall the following lemma of \cite{Murphy}.
\begin{lemma}\label{L1}
Suppose that $\mathcal A$ is a unital $C^*$-algebra with a unit $1_{\mathcal A}$.
\begin{enumerate}
\item[(1)]
If $a\in {\mathcal A}_+$ with $\|a\|<\frac{1}{2}$, then $1_{\mathcal A}-a$ is invertible.
\item[(2)]
If $a,b\in{\mathcal A}_+$ and $ab=ba$, then $0_{\mathcal A}\preceq ab$.
\item[(3)]
If $a,b\in{\mathcal A}_h$ and $c\in {\mathcal A}_+'$, then
$a\preceq b$ deduces $ca\preceq cb$, where ${\mathcal A}_+'={\mathcal A}_+\cap {\mathcal A}'$.
\end{enumerate}
\end{lemma}

\begin{theorem}\label{Th2}
Let $(X,\mathcal{A},d)$ be a complete $C^*$-algebra-valued metric space.
Suppose that the mappings $F\colon X\times X\rightarrow X$ and
$g\colon  X\rightarrow X$ satisfy the following condition
\begin{equation}\label{2-2}
d(F(x,y),F(u,v))\preceq ad(F(x,y),gx)+bd(F(u,v),gu), \ \mbox{for\ any}\ x,y,u,v\in X,
\end{equation}
where $a,b\in \mathcal{A}_{+}'$ with $\|a\|+\|b\|<1$.
If $F(X\times X)\subseteq g(X)$ and $g(X)$ is complete in $X$, then $F$ and $g$ have a
coupled coincidence point. Moreover, if $F$ and $g$ are w-compatible, then they have unique common
coupled fixed point in $X$.
\end{theorem}

\begin{proof}
Similar to Theorem \ref{Th1}, construct two sequences $\{x_n\}$ and $\{y_n\}$ in $X$ such that $gx_{n+1}=F(x_n,y_n)$ and $gy_{n+1}=F(y_n,x_n)$. Then by applying (\ref{2-2}) we have
\begin{eqnarray*}
    \begin{array}{rcl}
(1_{\mathcal A}-b)d(gx_n,gx_{n+1})&\preceq&ad(gx_n,gx_{n-1}),\\[5pt]
(1_{\mathcal A}-b)d(gy_n,gy_{n+1})&\preceq&ad(gy_n,gy_{n-1}).
\end{array}
 \end{eqnarray*}
Since $a,b\in \mathcal{A}_{+}'$ with $\|a\|+\|b\|<1$, we have
$1_{\mathcal A}-b$ is invertible and $(1_{\mathcal A}-b)^{-1}a\in \mathcal{A}_{+}'$. Therefore
\begin{eqnarray*}
    \begin{array}{rcl}
d(gx_n,gx_{n+1})&\preceq&(1_{\mathcal A}-b)^{-1}ad(gx_n,gx_{n-1}),\\[5pt]
d(gy_n,gy_{n+1})&\preceq&(1_{\mathcal A}-b)^{-1}ad(gy_n,gy_{n-1}).
\end{array}
 \end{eqnarray*}
Then
\begin{eqnarray*}
    \begin{array}{rcl}
\|d(gx_n,gx_{n+1})\|&\preceq&\|(1_{\mathcal A}-b)^{-1}a\|\|d(gx_n,gx_{n-1})\|,\\[5pt]
\|d(gy_n,gy_{n+1})\|&\preceq&\|(1_{\mathcal A}-b)^{-1}a\|\|d(gy_n,gy_{n-1})\|.
\end{array}
 \end{eqnarray*}
It follows from the fact
\begin{eqnarray*}
\|(1_{\mathcal A}-b)^{-1}a\|\leq \|(1_{\mathcal A}-b)^{-1}\|\|a\|
\leq\sum\limits_{k=0}^\infty\|b\|^k\|a\|=\frac{\|a\|}{1-\|b\|}<1
 \end{eqnarray*}
that $\{gx_n\}$ and $\{gy_n\}$ are Cauchy sequences in $g(X)$ and therefore by the completeness of $g(X)$, there are $x,y\in X$ such that $\lim\limits_{n\rightarrow\infty}gx_n=gx$ and $\lim\limits_{n\rightarrow\infty}gy_n=gy$.
Since
\begin{eqnarray*}
    \begin{array}{rcl}
d(F(x,y),gx)&\preceq&d(gx_{n+1},F(x,y))+d(gx_{n+1},gx)\\[5pt]
&=&d(F(x_n,y_n),F(x,y))+d(gx_{n+1},gx)\\[5pt]
&\preceq&ad(F(x_n,y_n),gx_n)+bd(F(x,y),gx)+d(gx_{n+1},gx)\\[5pt]
&\preceq&ad(gx_{n+1},gx_n)+bd(F(x,y),gx)+d(gx_{n+1},gx),
\end{array}
 \end{eqnarray*}
which implies that
\begin{eqnarray*}
d(F(x,y),gx)\preceq (1-b)^{-1}ad(gx_{n+1},gx_n)+(1-b)^{-1}d(gx_{n+1},gx).
\end{eqnarray*}
Then $d(F(x,y),gx)=0_{\mathcal A}$ or equivalently $F(x,y)=gx$.
Similarly, one can obtain $F(y,x)=gy$.

Now if $(x',y')$ is another coupled coincidence point of $F$ and $g$, then according to (\ref{2-2}), we obtain
\begin{eqnarray*}
    \begin{array}{rcl}
d(gx',gx)&\preceq&d(F(x',y'),F(x,y))\\[5pt]
&\preceq&ad(F(x',y'),gx')+bd(F(x,y),gx)=0_{\mathcal A},
\end{array}
 \end{eqnarray*}
and
\begin{eqnarray*}
    \begin{array}{rcl}
d(gy',gy)&\preceq&d(F(y',x'),F(y,x))\\[5pt]
&\preceq&ad(F(y',x'),gy')+bd(F(y,x),gy)=0_{\mathcal A},
\end{array}
 \end{eqnarray*}
which implies that $gx'=gx$ and $gy'=gy$. Similarly, we have
$gx'=gy$ and $gy'=gx$. Hence $F$ and $g$ have a unique coupled point of coincidence $(gx,gx)$. Moreover, we can show that $F$ and $g$ have a unique common coupled fixed point.
\end{proof}

It is worth noting that when the contractive elements in Theorem \ref{Th2} are equal we have the following corollary.

\begin{corollary}
Let $(X,\mathcal{A},d)$ be a complete $C^*$-algebra-valued metric space.
Suppose that the mappings $F\colon X\times X\rightarrow X$ and $g\colon X\rightarrow X$ satisfy the following condition
\begin{equation*}
d(F(x,y),F(u,v))\preceq ad(F(x,y),gx)+ad(F(u,v),gu), \ \mbox{for\ any}\ x,y,u,v\in X,
\end{equation*}
where $a\in \mathcal{A}_{+}'$ with $\|a\|<\frac{1}{2}$.
If $F(X\times X)\subseteq g(X)$ and $g(X)$ is complete in $X$, then $F$ and $g$ have a
coupled coincidence point.
\end{corollary}

\begin{corollary}
Let $(X,\mathcal{A},d)$ be a complete $C^*$-algebra-valued metric space.
Suppose that the mapping $F\colon X\times X\rightarrow X$ satisfies the following condition
\begin{equation*}
d(F(x,y),F(u,v))\preceq ad(F(x,y),x)+bd(F(u,v),u), \ \mbox{for\ any}\ x,y,u,v\in X,
\end{equation*}
where $a,b\in \mathcal{A}_{+}'$ with $\|a\|+\|b\|<1$.
Then $F$ has a
unique coupled fixed point.
\end{corollary}

\begin{theorem}
Let $(X,\mathcal{A},d)$ be a complete $C^*$-algebra-valued metric space.
Suppose that the mappings $F\colon X\times X\rightarrow X$ and $g\colon X\rightarrow X$ satisfy the following condition
\begin{equation}\label{2-3}
d(F(x,y),F(u,v))\preceq ad(F(x,y),gu)+bd(F(u,v),gx), \ \mbox{for\ any}\ x,y,u,v\in X,
\end{equation}
where $a,b\in \mathcal{A}_{+}'$ with $\|a\|+\|b\|<1$.
If $F(X\times X)\subseteq g(X)$ and $g(X)$ is complete in $X$, then $F$ and $g$ have a
coupled coincidence point. Moreover, if $F$ and $g$ are w-compatible, then they have unique common
coupled fixed point in $X$.
\end{theorem}

\begin{proof}
Following similar process given in Theorem \ref{Th1}, we construct two sequences $\{gx_n\}$ and $\{gy_n\}$ in $X$ such that
$gx_{n+1}=F(x_n,y_n)$ and $gy_{n+1}=F(y_n,x_n)$.
Now, from (\ref{2-3}), we have
\begin{eqnarray*}
    \begin{array}{rcl}
d(gx_n,gx_{n+1})&=&d(F(x_{n-1},y_{n-1}),F(x_n,y_n))\\[5pt]
      &\preceq&ad(F(x_{n-1},y_{n-1}),gx_n)+bd(F(x_n,y_n),gx_{n-1})\\[5pt]
      &\preceq&bd(gx_{n+1},gx_{n-1})\\[5pt]
      &\preceq&bd(gx_{n+1},gx_{n})+bd(gx_{n},gx_{n-1}),
 \end{array}
\end{eqnarray*}
from which it follows
\begin{eqnarray}\label{2-3-1}
(1_{\mathcal A}-b)d(gx_n,gx_{n+1})\preceq bd(gx_{n},gx_{n-1}).
\end{eqnarray}
Because of the symmetry in (\ref{2-3}),
\begin{eqnarray*}
    \begin{array}{rcl}
d(gx_{n+1},gx_n)&=&d(F(x_n,y_n),F(x_{n-1},y_{n-1}))\\[5pt]
      &\preceq&ad(F(x_n,y_n),gx_{n-1})+bd(F(x_{n-1},y_{n-1}),gx_n)\\[5pt]
      &\preceq&ad(gx_{n+1},gx_{n-1})\\[5pt]
      &\preceq&ad(gx_{n+1},gx_{n})+ad(gx_{n},gx_{n-1}),
 \end{array}
\end{eqnarray*}
that is,
\begin{eqnarray}\label{2-3-2}
(1_{\mathcal A}-a)d(gx_n,gx_{n+1})\preceq ad(gx_{n},gx_{n-1}).
\end{eqnarray}
Now, from (\ref{2-3-1}) and (\ref{2-3-2}) we obtain
\begin{eqnarray*}
\left(1_{\mathcal A}-\frac{a+b}{2}\right)d(gx_n,gx_{n+1})\preceq \frac{a+b}{2}d(gx_{n},gx_{n-1}).
\end{eqnarray*}
Since $a,b\in \mathcal{A}_{+}'$ with $\|a+b\|\leq\|a\|+\|b\|<1$, then $\left(1_{\mathcal A}-\frac{a+b}{2}\right)^{-1}\in \mathcal{A}_{+}'$, which together with Lemma \ref{L1} (3) yields that
\begin{eqnarray*}
d(gx_n,gx_{n+1})\preceq \left(1_{\mathcal A}-\frac{a+b}{2}\right)^{-1}\frac{a+b}{2}d(gx_{n},gx_{n-1}).
\end{eqnarray*}
Let $t=\left(1_{\mathcal A}-\frac{a+b}{2}\right)^{-1} \frac{a+b}{2}$, then $\|t\|=\|\left(1_{\mathcal A}-\frac{a+b}{2}\right)^{-1} \frac{a+b}{2}\|<1$.
The same argument in Theorem \ref{Th2} tells that
$\{gx_n\}$ is a Cauchy sequence in $g(X)$.
Similarly, we can show $\{gy_n\}$ is also a Cauchy sequence in $g(X)$.
Therefore by the completeness of $g(X)$, there are $x,y\in X$ such that $\lim\limits_{n\rightarrow\infty}gx_n=gx$ and $\lim\limits_{n\rightarrow\infty}gy_n=gy$.
Now, we prove that $F(x,y)=gx$ and $F(y,x)=gy$. For that we have
\begin{eqnarray*}
    \begin{array}{rcl}
d(F(x,y),gx)&\preceq&d(gx_{n+1},F(x,y))+d(gx_{n+1},gx)\\[5pt]
&=&d(F(x_n,y_n),F(x,y))+d(gx_{n+1},gx)\\[5pt]
&\preceq&ad(F(x_n,y_n),gx)+bd(F(x,y),gx_n)+d(gx_{n+1},gx)\\[5pt]
&\preceq&ad(gx_{n+1},gx)+bd(F(x,y),gx_n)+d(gx_{n+1},gx),
\end{array}
 \end{eqnarray*}
and then
\begin{eqnarray*}
\|d(F(x,y),gx)\|\leq \|a\|\|d(gx_{n+1},gx)\|+\|b\|\|d(F(x,y),gx_n)\|+\|d(gx_{n+1},gx)\|.
 \end{eqnarray*}
By the continuity of the metric and the norm, we know
\begin{eqnarray*}
\|d(F(x,y),gx)\|\leq \|b\|\|d(F(x,y),gx)\|.
 \end{eqnarray*}
It follows from the fact $\|b\|<1$ that $\|d(F(x,y),gx)\|=0$. Thus $F(x,y)=gx$. Similarly, $F(y,x)=gy$. Hence $(x,y)$ is a coupled coincidence
point of $F$ and $g$.
The same reasoning that in Theorem \ref{Th2} tells us that $F$ and $g$ have unique common
coupled fixed point in $X$.
\end{proof}

\begin{corollary}
Let $(X,\mathcal{A},d)$ be a complete $C^*$-algebra-valued metric space.
Suppose that the mapping $F\colon X\times X\rightarrow X$ satisfies the following condition
\begin{equation*}
d(F(x,y),F(u,v))\preceq ad(F(x,y),gu)+ad(F(u,v),gx), \ \mbox{for\ any}\ x,y,u,v\in X,
\end{equation*}
where $a\in \mathcal{A}_{+}'$ with $\|a\|<\frac{1}{2}$.
If $F(X\times X)\subseteq g(X)$ and $g(X)$ is complete in $X$, then $F$ and $g$ have a
coupled coincidence point.
\end{corollary}
\begin{corollary}
Let $(X,\mathcal{A},d)$ be a complete $C^*$-algebra-valued metric space.
Suppose that the mapping $F\colon X\times X\rightarrow X$ satisfies the following condition
\begin{equation*}
d(F(x,y),F(u,v))\preceq ad(F(x,y),u)+bd(F(u,v),x), \ \mbox{for\ any}\ x,y,u,v\in X,
\end{equation*}
where $a,b\in \mathcal{A}_{+}'$ with $\|a\|+\|b\|<1$.
Then $F$ has a
unique coupled fixed point.
\end{corollary}

Coupled fixed point theorems in partially ordered metric spaces are widely investigated and have been found various applications in integral equations and periodic boundary value problem (see \cite{Berinde,Bhaskar,Hussain} and reference therein).
The coupled fixed point theorems proved here pave the way for an application on complete $C^*$-algebra-valued metric spaces to prove the existence and uniqueness of a solution for a Fredholm nonlinear integral equation.

Consider the integral equation
\begin{eqnarray}\label{3}
x(t)=\int_E \Big(K_1(t,s)+K_2(t,s)\Big)\Big(f(s,x(s))+g(s,x(s))\Big)ds+h(t), \ \ t\in E,
\end{eqnarray}
where $E$ is a Lebesgue measurable set and $m(E)<\infty$.

In what follows, we always let
$X=L^\infty(E)$ denote the class of essentially bounded measurable functions on $E$, where $E$ is a Lebesgue measurable set such that $m(E)<\infty$.

Now, we consider the functions $K_1$, $K_2$, $f$, $g$ fulfill the following assumptions:
\begin{enumerate}
\item[(i)]
$K_1\colon E\times E\rightarrow [0,+\infty)$, $K_2\colon E\times E\rightarrow (-\infty,0]$,
$f,g\colon E\times\Bbb{R}\rightarrow\Bbb{R}$ are integrable, and $h\in L^\infty(E)$;
\item[(ii)]
there exist $k\in (0,\frac{1}{2})$ such that
$$0\leq f(t,x)-f(t,y)\leq k(x-y)$$ and
$$-k(x-y)\leq g(t,x)-g(t,y)\leq 0$$
for $t\in E$ and $x,y\in\Bbb{R}$;
\item[(iii)]
$\sup\limits_{t\in E}\int_E (K_1(t,s)-K_2(t,s))ds\leq 1$.
\end{enumerate}

\begin{theorem}
Suppose that assumptions (i)-(iii) hold.
Then the integral equation (\ref{3}) has a unique solution in $L^\infty(E)$.
\end{theorem}

\begin{proof}
Let $X=L^\infty(E)$ and $B(L^2(E))$ be the set of bounded linear operators on a Hilbert space $L^2(E)$.
We endow $X$ with the cone metric
 $d\colon X\times X\rightarrow B(L^2(E))$ defined by
$$d(f,g)=M_{|f-g|},$$
where $M_{|f-g|}$ is the multiplication operator on $L^2(E)$. It is clear that $(X,B(L^2(E)),d)$ is a complete $C^*$-algebra-valued metric space.

Define the self-mapping $F\colon X\times X\rightarrow X$ by
$$F(x,y)(t)=\int_E K_1(t,s)\Big(f(s,x(s))+g(s,y(s))\Big)ds
+K_2(t,s)\Big(f(s,y(s))+g(s,x(s))\Big)ds+h(t),$$
for all $t\in E$.

Now, we have
$$d(F(x,y),F(u,v))=M_{|F(x,y)-F(u,v)|}.$$
Let us first evaluate the following expression:
\begin{eqnarray*}
    \begin{array}{rcl}
&&|(F(x,y)-F(u,v))(t)|\\[6pt]
&=&|\int_E K_1(t,s)\Big(f(s,x(s))+g(s,y(s))\Big)ds
+\int_E K_2(t,s)\Big(f(s,y(s))+g(s,x(s))\Big)ds\\[6pt]
&&-
\int_E K_1(t,s)\Big(f(s,u(s))+g(s,v(s))\Big)ds
-\int_E K_2(t,s)\Big(f(s,v(s))+g(s,u(s))\Big)ds|\\[6pt]
&=&|\int_E K_1(t,s)\Big(f(s,x(s))-f(s,u(s))+g(s,y(s))-g(s,v(s))\Big)ds|\\[6pt]
&&+|\int_E K_2(t,s)\Big(f(s,y(s))-f(s,v(s))+g(s,x(s))-g(s,u(s))\Big)ds|\\[6pt]
&\leq&\int_E K_1(t,s)|f(s,x(s))-f(s,u(s))+g(s,y(s))-g(s,v(s))|ds\\[6pt]
&&-\int_E
K_2(t,s)|f(s,y(s))-f(s,v(s))+g(s,x(s))-g(s,u(s))|ds\\[6pt]
&\leq&\sup\limits_{s\in E}[k|x(s)-u(s)|+k|y(s)-v(s)|]\int_E (K_1(t,s)-K_2(t,s))ds\\[5pt]
&\leq&[k\|x-u\|_\infty+k\|y-v\|_\infty]\sup\limits_{t\in E}\int_E (K_1(t,s)-K_2(t,s))ds\\[6pt]
&\leq&k\|x-u\|_\infty+k\|y-v\|_\infty.
 \end{array}
\end{eqnarray*}
Therefore, we have
\begin{eqnarray*}
    \begin{array}{rcl}
\|d(F(x,y),F(u,v))\|&=&\|M_{|F(x,y)-F(u,v)|}\|\\[6pt]
&=&\sup\limits_{\|\varphi\|=1}(M_{|F(x,y)-F(u,v)|}\varphi,\varphi)\\[5pt]
&=&\sup\limits_{\|\varphi\|=1}\int_E |(F(x,y)-F(u,v))(t)|\varphi(t)\overline{\varphi(t)}dt\\[6pt]
&\leq&\sup\limits_{\|\varphi\|=1}
\int_E |\varphi(t)|^2dt (k\|x-u\|_\infty+k\|y-v\|_\infty)\\[6pt]
&\leq&k\|x-u\|_\infty+k\|y-v\|_\infty.
 \end{array}
\end{eqnarray*}
Set $a=\sqrt{k}1_{B(L^2(E))}$, then $a\in B(L^2(E))$ and $\|a\|=|\sqrt{k}|<\frac{1}{\sqrt{2}}$.
Hence, applying our Corollary \ref{corollary1}, we get the desired result.
\end{proof}


\begin{thebibliography}{}

\bibitem{Abbas} Abbas, M, Khan, MA, Radenovic, S.
Common coupled fixed point theorems in cone metric spaces for w-compatible mappings, Appl. Math. Comput.
217, 195-202 (2010)


\bibitem{Aydi} Aydi, H, Samet, B, Vetro, C. Coupled fixed point results in cone metric spaces for compatible mappings, Fixed Point Theory Appl. 2011, 15 (2011)

\bibitem{Beg} Beg, I, Butt, AR. Coupled fixed points of set valued mappings in partially ordered metric spaces, J. Nonlinear Sci. Appl. 3, 179-185 (2010)
\bibitem{Berinde} Berinde, V. Coupled fixed point theorems for $\phi$-contractive mixed monotone mappings in partially ordered metric spaces, Nonlinear Anal. 75, 3218-3228 (2012)
\bibitem{Bhaskar} Bhaskar, TG, Lakshmikantham, V. Fixed point theorems in partially ordered metric spaces and applications, Nonlinear Anal. 65, 1379-1393 (2006)
\bibitem{Chen} Chen, J, Huang, X. Coupled fixed point theorems for compatible mappings in partially ordered $G$-metric spaces, J. Nonlinear Sci. Appl. 8, 130-141 (2015)
\bibitem{Guo} Guo, D, Lakshmikantham, V. Coupled fixed points of nonlinear operators with applications, Nonlinear Anal. 11, 623-632 (1987)
\bibitem{Huang} Huang, L, Zhang, X.
    Cone metric spaces and fixed point theorems of contractive mappings, J. Math. Anal. Appl. 332, 1468-1476 (2007)
\bibitem{Hussain} Hussain, N, Salimi, P, Al-Mezel S. Coupled fixed point results on quasi-Banach spaces with application to a system of integral equations, Fixed Point Theory Appl. 2013, 261 (2013)

\bibitem{Jleli} Jleli, M, Samet, B. On positive solutions for a class of singular nonlinear fractional differential equations. Bound. Value Probl. 2012, 73 (2012)

\bibitem{Karapinar10} Karapinar, E. Couple fixed point theorems for nonlinear contractions in cone metric spaces, Comput. Math. Appl. 59, 3656-3668 (2010)

\bibitem{Luong} Luong, NV, Thuan, NX. Coupled fixed points in partially ordered metric spaces and application, Nonlinear Anal. 74, 983-992 (2011)

\bibitem{Ma} Ma, ZH, Jiang, LN, Sun, HK. $C^*$-algebra-valued metric spaces and related fixed point theorems, Fixed Point Theory Appl. 2014, 206 (2014)
\bibitem{Malhotra} Malhotra, N, Bansal, B. Some common coupled fixed point theorems for generalized contraction in $b$-metric spaces, J. Nonlinear Sci. Appl. 8, 8-16 (2015)
\bibitem{Mihet} Gupta, V, Shatanawi, W, Kanwar, A. Coupled fixed point theorems employing CLR-property on V-fuzzy metric spaces, Mathematics 8, 404 (2020)
\bibitem{Murphy} Murphy, GJ. $C^*$-algebras and operator theory. Academic Press, London (1990)
\bibitem{Nieto07} Nieto, JJ, Rodriguez-Lopez, R. Existence and uniqueness of fixed point in partially ordered sets and applications to ordinary differential equations, Acta Math. Sin. (Engl. Ser.) 23, 2205-2212 (2007)
\bibitem{Ran} Ran, ACM, Reurings, MCB. A fixed point theorem in partially ordered sets and some applications to matrix equations, Proc. Amer. Math. Soc. 132, 1435-1443 (2004)
\bibitem{Sabetghadam} Sabetghadam, F, Masiha, HP, Sanatpour, AH. Some coupled fixed point theorems in cone metric spaces, Fixed Point Theory Appl. 2009, 125426 (2009)

\bibitem{Samet} Samet, B, Vetro, C. Coupled fixed point theorems for multi-valued nonlinear contraction mappings in partially ordered metric spaces, Nonlinear Anal. 74, 4260-4268 (2011)
\bibitem{Xin} Xin, QL, Jiang, LN. Fixed-point theorems for mappings satisfying the ordered contractive condition on noncommutative spaces. Fixed Point Theory Appl. 2014, 30 (2014)






 \end{thebibliography}
  \end{document}